\documentclass[12pt]{article}

\usepackage{amsmath,amsthm}
\usepackage{amssymb}
\usepackage{hyperref}

\begin{document}

\title{Truncated Bernoulli-Carlitz and truncated Cauchy-Carlitz numbers}  

\author{
Takao Komatsu\thanks{
The research of Takao Komatsu was supported in part by the grant of Wuhan University and by the grant of Hubei Provincial Experts Program. 
}\\ 
\small School of Mathematics and Statistics\\
\small Wuhan University\\
\small Wuhan 430072 China\\
\small \texttt{komatsu@whu.edu.cn}
}

\date{
%\small Submitted: November 1, 2016;  Accepted: December 2, 2016.\\
\small MR Subject Classifications: Primary 11R58; Secondary 11T55, 11B68, 11B73, 11B75, 05A15, 05A19.
}

\maketitle

\def\stf#1#2{\left[#1\atop#2\right]} 
\def\sts#1#2{\left\{#1\atop#2\right\}}

\newtheorem{theorem}{Theorem}
\newtheorem{Prop}{Proposition}
\newtheorem{Cor}{Corollary}
\newtheorem{Lem}{Lemma}

\begin{abstract}
In this paper, we define the truncated Bernoulli-Carlitz numbers and the truncated Cauchy-Carlitz numbers 
as analogues of hypergeometric Bernoulli numbers and hypergeometric Cauchy numbers, 
and as extensions of Bernoulli-Carlitz numbers and the Cauchy-Carlitz numbers. 
 These numbers can be expressed explicitly in terms of incomplete Stirling-Carlitz numbers.   
\end{abstract}

\maketitle

\section{Introduction}   

For $N\ge 1$, hypergeometric Bernoulli numbers $B_{N,n}$ (\cite{HN1,HN2,Kamano}) are defined by the generating function 
\begin{equation}  
\frac{1}{{}_1 F_1(1;N+1;x)}=\frac{x^N/N!}{e^t-\sum_{n=0}^{N-1}x^n/n!}=\sum_{n=0}^\infty B_{N,n}\frac{x^n}{n!}\,,
\label{def:hgbernoulli}
\end{equation}  
where 
$$
{}_1 F_1(a;b;z)=\sum_{n=0}^\infty\frac{(a)^{(n)}}{(b)^{(n)}}\frac{z^n}{n!}
$$ 
is the confluent hypergeometric function with $(x)^{(n)}=x(x+1)\cdots(x+n-1)$ ($n\ge 1$) and $(x)^{(0)}=1$.  
When $N=1$, $B_n=B_{1,n}$ are classical Bernoulli numbers defined by 
$$
\frac{t}{e^t-1}=\sum_{n=0}^\infty B_n\frac{t^n}{n!}\,. 
$$ 
In addition, hypergeometric Cauchy numbers $c_{N,n}$ (see \cite{Ko3}) are defined by 
\begin{equation} 
\frac{1}{{}_2 F_1(1,N;N+1;-x)}=\frac{(-1)^{N-1}x^N/N}{\log(1+t)-\sum_{n=1}^{N-1}(-1)^{N-1}x^n/n}=\sum_{n=0}^\infty c_{N,n}\frac{x^n}{n!}\,,
\label{def:hgcauchy}
\end{equation}  
where 
$$
{}_2 F_1(a,b;c;z)=\sum_{n=0}^\infty\frac{(a)^{(n)}(b)^{(n)}}{(c)^{(n)}}\frac{z^n}{n!}
$$ 
is the Gauss hypergeometric function.  
When $N=1$, $c_n=c_{1,n}$ are classical Cauchy numbers defined by 
$$
\frac{t}{\log(1+t)}=\sum_{n=0}^\infty c_n\frac{t^n}{n!}\,. 
$$
On the other hand,  L. Carlitz (\cite{Car35}) introduced analogues of Bernoulli numbers for the rational function (finite) field $K=\mathbb F_r(T)$, which are called Bernoulli-Carlitz numbers now. Bernoulli-Carlitz numbers have been studied since then (e.g., see \cite{Car37,Car40,Gekeler,JKS,Rodriguez}).   According to the notations by Goss \cite{Goss}, Bernoulli-Carlitz numbers are defined by 
\begin{equation}  
\frac{x}{e_C(x)}=\sum_{n=0}^\infty\frac{BC_n}{\Pi(n)}x^n\,.
\label{def:bercarlitz} 
\end{equation} 
Here, $e_C(x)$ are the Carlitz exponential defined by 
\begin{equation} 
e_C(x)=\sum_{i=0}^\infty\frac{x^{r^i}}{D_i}\,, 
\label{carlitzexp}
\end{equation}  
where $D_i=[i][i-1]^r\cdots [1]^{r^{i-1}}$ ($i\ge 1$) with $D_0=1$, and $[i]=T^{r^i}-T$.  
The Carlitz factorial $\Pi(i)$ is defined by 
\begin{equation}  
\Pi(i)=\prod_{j=0}^m D_j^{c_j}
\label{carlitzfac}
\end{equation} 
for a non-negative integer $i$ with $r$-ary expansion: 
\begin{equation}  
i=\sum_{j=0}^m c_j r^j\quad (0\le c_j<r)\,. 
\label{r-expansion}
\end{equation}  

As analogues of the classical Cauchy numbers $c_n$, Cauchy-Carlitz numbers $CC_n$ (\cite{KK}) are introduced as 
\begin{equation}  
\frac{x}{\log_C(x)}=\sum_{n=0}^\infty\frac{CC_n}{\Pi(n)}x^n\,.  
\label{def:caucarlitz} 
\end{equation}  
Here, $\log_C(x)$ is the Carlitz logarithm  defined by 
\begin{equation} 
\log_C(x)=\sum_{i=0}^\infty(-1)^i\frac{x^{r^i}}{L_i}\,, 
\label{carlitzlog}
\end{equation}  
where $L_i=[i][i-1]\cdots [1]$ ($i\ge 1$) with $L_0=1$.  

In \cite{KK}, Bernoulli-Carlitz numbers and Cauchy-Carlitz numbers are expressed explicitly 
by using the Stirling-Carlitz numbers of the second kind and of the first kind, respectively.  
These properties are the extensions that Bernoulli numbers and Cauchy numbers are expressed explicitly 
by using the Stirling numbers of the second kind and of the first kind, respectively.  

In this paper, we define the truncated Bernoulli-Carlitz numbers and the truncated Cauchy-Carlitz numbers 
as analogues of hypergeometric Bernoulli numbers and hypergeometric Cauchy numbers, 
and as extensions of Bernoulli-Carlitz numbers and the Cauchy-Carlitz numbers. 
 These numbers can be expressed explicitly in terms of incomplete Stirling-Carlitz numbers.

\section{Preliminaries}  

For $N\ge 1$, define the truncated Bernoulli-Carlitz numbers $BC_{N,n}$ and the truncated Cauchy-Carlitz numbers $CC_{N,n}$ by 
\begin{equation}  
\frac{x^{r^N}/D_N}{e_C(x)-\sum_{i=0}^{N-1}x^{r^i}/D_i}=\sum_{n=0}^\infty\frac{BC_{N,n}}{\Pi(n)}x^n
\label{def:hgbercarlitz} 
\end{equation}  
and 
\begin{equation}  
\frac{(-1)^N x^{r^N}/L_N}{\log_C(x)-\sum_{i=0}^{N-1}(-1)^i x^{r^i}/L_i}=\sum_{n=0}^\infty\frac{CC_{N,n}}{\Pi(n)}x^n\,, 
\label{def:hgcaucarlitz} 
\end{equation}  
respectively.  
When $N=0$, $BC_n=BC_{0,n}$ and $CC_n=CC_{0,n}$ are the original Bernoulli-Carlitz numbers and Cauchy-Carlitz numbers, 
respectively.  As the concept of these definitions in (\ref{def:hgbercarlitz}) and (\ref{def:hgcaucarlitz} ) in function fields 
are the same as (\ref{def:hgbernoulli}) and (\ref{def:hgcauchy}) in complex numbers,     
the numbers $BC_{N,n}$ and $CC_{N,n}$ could be called the hypergeometric Bernoulli-Carlitz numbers and the 
hypergeometric Cauchy-Carlitz numbers, respectively. 
However, the generating functions of (\ref{def:hgbercarlitz}) and (\ref{def:hgcaucarlitz} ) are not related to the existing  
Carlitz hypergeometric functions (e.g., see \cite{Koc1,Thakur}).

\section{Hasse-Teichm\"uller derivatives}  

Let $\mathbb{F}$ be a field (of any characterstic), $\mathbb{F}((z))$ be the field of Laurent series in $z$, and $\mathbb{F}[[z]]$ be the ring of formal power series.   
The Hasse-Teichm\"uller derivative $H^{(n)}$ of order $n$ is defined by 
$$
H^{(n)}\left(\sum_{m=R}^{\infty} a_m z^m\right)
=\sum_{m=R}^{\infty} a_m \binom{m}{n}z^{m-n}
$$
for $\sum_{m=R}^{\infty} a_m z^m\in \mathbb{F}((z))$, 
where $R$ is an integer and $a_m\in\mathbb{F}$ for any $m\geq R$. 

The Hasse-Teichm\"uller derivatives satisfy the product rule \cite{Teich}, the quotient rule \cite{GN} and the chain rule \cite{Hasse}. 
One of the product rules can be described as follows.  
\begin{Lem}  
For $f_i\in\mathbb F[[z]]$ ($i=1,\dots,k$) with $k\ge 2$ and for $n\ge 1$, we have 
$$
H^{(n)}(f_1\cdots f_k)=\sum_{i_1,\dots,i_k\ge 0\atop i_1+\cdots+i_k=n}H^{(i_1)}(f_1)\cdots H^{(i_k)}(f_k)\,. 
$$ 
\label{productrule2}
\end{Lem} 

The quotient rules can be described as follows.  

\begin{Lem}  
For $f\in\mathbb F[[z]]\backslash \{0\}$ and $n\ge 1$,  
we have 
\begin{align} 
H^{(n)}\left(\frac{1}{f}\right)&=\sum_{k=1}^n\frac{(-1)^k}{f^{k+1}}\sum_{i_1,\dots,i_k\ge 1\atop i_1+\cdots+i_k=n}H^{(i_1)}(f)\cdots H^{(i_k)}(f)
\label{quotientrule1}\\ 
&=\sum_{k=1}^n\binom{n+1}{k+1}\frac{(-1)^k}{f^{k+1}}\sum_{i_1,\dots,i_k\ge 0\atop i_1+\cdots+i_k=n}H^{(i_1)}(f)\cdots H^{(i_k)}(f)\,.
\label{quotientrule2} 
\end{align}   
\label{quotientrules}
\end{Lem}  

By using the Hasse-Teichm\"uller derivative of order $n$, we shall obtain some explicit expressions of the hypergeometric Bernoulli-Carlitz numbers $BC_{N,n}$ and hypergeometric Cauchy numbers $CC_{N,n}$, respectively.

\begin{theorem}  
For $n\ge 1$, 
$$
BC_{N,n}=\Pi(n)\sum_{k=1}^{n}(-D_N)^k\sum_{i_1,\dots,i_k\ge 1\atop r^{N+i_1}+\cdots+r^{N+i_k}=n+k r^N}\frac{1}{D_{N+i_1}\cdots D_{N+i_k}}\,.
$$ 
\label{th_hgbernoullicarlitz1} 
\end{theorem}  

\noindent 
{\it Remark.}  
It is clear that $BC_{N,n}=0$ if $r\nmid n$ or $r^N(r-1)>n$.  
When $N=0$, we have 
$$
BC_n=\Pi(n)\sum_{k=1}^{n}(-1)^k\sum_{i_1,\dots,i_k\ge 1\atop r^{i_1}+\cdots+r^{i_k}=n+k}\frac{1}{D_{i_1}\cdots D_{i_k}}\,
$$ 
which is Theorem 4.2 in \cite{JKS}.  

\begin{proof}  
Put 
$$
h:=\dfrac{\sum_{i=N}^\infty\dfrac{x^{r^i}}{D_i}}{\dfrac{x^{r^N}}{D_N}}=\sum_{j=0}^\infty\frac{D_N}{D_{N+j}}x^{r^{N+j}-r^N}\,. 
$$ 
Note that 
\begin{align*} 
\left. H^{(e)}(h)\right|_{x=0}&=\left.\sum_{j=0}^\infty\frac{D_N}{D_{N+j}}\binom{r^{N+j}-r^N}{e}x^{r^{N+j}-r^N-e}\right|_{x=0}\\
&=\begin{cases} 
\dfrac{D_N}{D_{N+i}}&\text{if $e=r^{N+i}-r^N$};\\ 
0&\text{otherwise}. 
\end{cases} 
\end{align*}  
Hence, by using Lemma \ref{quotientrules} (\ref{quotientrule1}), we have 
\begin{align*} 
\frac{BC_{N,n}}{\Pi(n)}&=\left.H^{(n)}\left(\frac{1}{h}\right)\right|_{x=0}\\
&=\sum_{k=1}^n\left.\frac{(-1)^k}{h^{k+1}}\right|_{x=0}\sum_{e_1,\dots,e_k\ge 1\atop e_1+\cdots+e_k=n}\left.H^{(e_1)}(h)\right|_{x=0}\cdots\left.H^{(e_k)}(h)\right|_{x=0}\\
&=\sum_{k=1}^n(-1)^k\sum_{i_1,\dots,i_k\ge 1\atop r^{N+i_1}+\cdots+r^{N+i_k}=n+k r^N}\frac{D_N}{D_{N+i_1}}\cdots\frac{D_N}{D_{N+i_k}}\\
&=\sum_{k=1}^n(-D_N)^k\sum_{i_1,\dots,i_k\ge 1\atop r^{N+i_1}+\cdots+r^{N+i_k}=n+k r^N}\frac{1}{D_{N+i_1}\cdots D_{N+i_k}}\,. 
\end{align*} 
\end{proof}   

\noindent 
{\bf Examples.} 
Let $r=3$ and $N=2$.  Then $BC_{2,n}=0$ if $18\nmid n$.  
When $n=18$, consider the set 
$$
S_k=\{(i_1,\dots,i_k)|i_1,\dots,i_k\ge 1,~3^{i_1+2}+\cdots+3^{i_k+2}=18+9 k\}\,.
$$ 
Then $S_1=\{(1)\}$, and $S_k$ is empty when $k\ge 2$ because $3^{i_1+2}+3^{i_2+2}\ge 54>36$.  
Hence, we obtain 
$$
BC_{2,18}=\Pi(18)(-D_2)\frac{1}{D_3}\,. 
$$ 
When $n=36$, consider the set 
$$
S_k=\{(i_1,\dots,i_k)|i_1,\dots,i_k\ge 1,~3^{i_1+2}+\cdots+3^{i_k+2}=36+9 k\}\,.
$$ 
Then, $S_k$ ($k=1$, $k\ge 3$) are empty because $3^{i_1+2}+3^{i_2+2}+3^{i_3+2}\ge 81>63$.  By $S_2=\{(1,1)\}$, we have 
$$
BC_{2,36}=\Pi(36)(-D_2)^2\frac{1}{D_3 D_3}=\Pi(36)\frac{D_2^2}{D_3^2}\,. 
$$ 
When $n=72$, consider the set 
$$
S_k=\{(i_1,\dots,i_k)|i_1,\dots,i_k\ge 1,~3^{i_1+2}+\cdots+3^{i_k+2}=72+9 k\}\,.
$$ 
Since $S_k$ is empty for $k=2,3$ and $k\ge 5$ and $S_1=\{(2)\}$ and $S_4=\{(1,1,1,1)\}$, we have 
$$
BC_{2,72}=\Pi(72)\left(\frac{-D_2}{D_4}+\frac{D_2^4}{D_3 D_3 D_3 D_3}\right)\,. 
$$ 
In fact, 
\begin{align*} 
\sum_{n=0}^\infty\frac{BC_{2,n}}{\Pi(n)}x^n&=\dfrac{\dfrac{x^9}{D_2}}{\sum_{i=2}^\infty\dfrac{x^{3^i}}{D_i}}\\
&=1-\frac{D_2}{D_3}x^{18}+\frac{D_2^2}{D_3^2}x^{36}-\frac{D_2^3}{D_3^3}x^{54}\\
&\quad+\left(\frac{D_2^4}{D_3^4}-\frac{D_2}{D_4}\right)x^{72}-\left(\frac{D_2^5}{D_3^5}-\frac{2 D_2^2}{D_3 D_4}\right)x^{72}+\cdots\,. 
\end{align*} 
\bigskip

We can express the hypergeometric Bernoulli-Carlitz numbers in terms of the binomial coefficients too. 
By using Lemma \ref{quotientrules} (\ref{quotientrule2}) 
instead of Lemma \ref{quotientrules} (\ref{quotientrule1}) 
in the proof of Theorem \ref{th_hgbernoullicarlitz1}, we obtain the following:  

\begin{Prop}  
For $n\ge 1$, 
$$
BC_{N,n}=\Pi(n)\sum_{k=1}^{n}\binom{n+1}{k+1}(-D_N)^k\sum_{i_1,\dots,i_k\ge 0\atop r^{N+i_1}+\cdots+r^{N+i_k}=n+k r^N}\frac{1}{D_{N+i_1}\cdots D_{N+i_k}}\,.
$$ 
\label{th_hgbernoullicarlitz2} 
\end{Prop} 

\noindent 
{\it Remark.}  
When $N=0$, we have 
$$
BC_n=\Pi(n)\sum_{k=1}^{n}\binom{n+1}{k+1}(-1)^k\sum_{i_1,\dots,i_k\ge 0\atop r^{i_1}+\cdots+r^{i_k}=n+k}\frac{1}{D_{i_1}\cdots D_{i_k}}\,
$$ 
which is Proposition 4.4 in \cite{JKS}.  

\noindent 
{\bf Example.}  
Let $r=3$ and $N=2$. When $n=18$, consider the set 
$$
S_k=\{(i_1,\dots,i_k)|i_1,\dots,i_k\ge 0,~3^{i_1+2}+\cdots+3^{i_k+2}=18+9 k\}\,.
$$ 
Since $S_1=\{(1)\}$, $S_2=\{(0,1),(1,0)\}$, $S_3=\{(0,0,1),(0,1,0),(1,0,0)\}$, $\dots$, $S_{18}=\{(\underbrace{0,\dots,0}_{17},1),\dots,(1,\underbrace{0,\dots,0}_{17})\}$, we have 
\begin{align*} 
\frac{BC_{2,18}}{\Pi(18)}&=-\binom{19}{2}\frac{D_2}{D_3}+\binom{19}{3}\frac{2 D_2^2}{D_2 D_3}-\binom{19}{4}\frac{3 D_2^3}{D_2^2 D_3}+\cdots+\binom{19}{19}\frac{18 D_2^{18}}{D_2^{17}D_3}\\
&=\frac{D_2}{D_3}\sum_{k=1}^{18}(-1)^k\binom{19}{k+1}k\\
&=\frac{D_2}{D_3}\left(\sum_{k=0}^{19}(-1)^{k-1}\binom{19}{k}(k-1)-1\right)\\
&=\frac{D_2}{D_3}\left(\sum_{k=0}^{19}(-1)^{k-1}\binom{19}{k}k+\sum_{k=0}^{19}(-1)^{k}\binom{19}{k}-1\right)\\
&=-\frac{D_2}{D_3}\,. 
\end{align*} 
\bigskip

Next, we shall give an explicit formula for hypergeometric Cauchy-Carlitz numbers. 

\begin{theorem}  
For $n\ge 1$, 
$$
CC_{N,n}=\Pi(n)\sum_{k=1}^{n}(-L_N)^k\sum_{i_1,\dots,i_k\ge 1\atop r^{N+i_1}+\cdots+r^{N+i_k}=n+k r^N}\frac{(-1)^{i_1+\cdots+i_k}}{L_{N+i_1}\cdots L_{N+i_k}}\,.
$$ 
\label{th_hgcauchycarlitz1} 
\end{theorem}  
\noindent 
{\it Remark.}  
It is clear that $CC_{N,n}=0$ if $r\nmid n$ or $r^N(r-1)>n$.  
When $N=1$, we have 
$$
CC_n=\Pi(n)\sum_{k=1}^{n}(-1)^k\sum_{i_1,\dots,i_k\ge 1\atop r^{i_1}+\cdots+r^{i_k}=n+k}\frac{(-1)^{i_1+\cdots+i_k}}{L_{i_1}\cdots L_{i_k}}\,
$$ 
which is Theorem 3 in \cite{KK}.  

\begin{proof}  
Put 
$$
h:=\dfrac{\sum_{i=N}^\infty(-1)^i\dfrac{x^{r^i}}{L_i}}{(-1)^N\dfrac{x^{r^N}}{L_N}}=\sum_{j=0}^\infty(-1)^j\frac{L_N}{L_{N+j}}x^{r^{N+j}-r^N}\,. 
$$ 
Note that 
\begin{align*} 
\left. H^{(e)}(h)\right|_{x=0}&=\left.\sum_{j=0}^\infty(-1)^j\frac{L_N}{L_{N+j}}\binom{r^{N+j}-r^N}{e}x^{r^{N+j}-r^N-e}\right|_{x=0}\\
&=\begin{cases} 
\dfrac{(-1)^i L_N}{L_{N+i}}&\text{if $e=r^{N+i}-r^N$};\\ 
0&\text{otherwise}. 
\end{cases} 
\end{align*}  
Hence, by using Lemma \ref{quotientrules} (\ref{quotientrule1}), we have 
\begin{align*} 
\frac{CC_{N,n}}{\Pi(n)}&=\left.H^{(n)}\left(\frac{1}{h}\right)\right|_{x=0}\\
&=\sum_{k=1}^n\left.\frac{(-1)^k}{h^{k+1}}\right|_{x=0}\sum_{e_1,\dots,e_k\ge 1\atop e_1+\cdots+e_k=n}\left.H^{(e_1)}(h)\right|_{x=0}\cdots\left.H^{(e_k)}(h)\right|_{x=0}\\
&=\sum_{k=1}^n(-1)^k\sum_{i_1,\dots,i_k\ge 1\atop r^{N+i_1}+\cdots+r^{N+i_k}=n+k r^N}\frac{(-1)^{i_1}L_N}{L_{N+i_1}}\cdots\frac{(-1)^{i_k}L_N}{L_{N+i_k}}\\
&=\sum_{k=1}^n(-L_N)^k\sum_{i_1,\dots,i_k\ge 1\atop r^{N+i_1}+\cdots+r^{N+i_k}=n+k r^N}\frac{(-1)^{i_1+\cdots+i_k}}{L_{N+i_1}\cdots L_{N+i_k}}\,. 
\end{align*} 
\end{proof}   

\noindent 
{\bf Example.} 
Let $r=3$ and $N=3$.  Then $CC_{3,n}=0$ if $54\nmid n$.  
When $n=270$, consider the set 
$$
S_k=\{(i_1,\dots,i_k)|i_1,\dots,i_k\ge 1,~3^{i_1+3}+\cdots+3^{i_k+3}=270+27 k\}\,.
$$ 
Then $S_2=\{(1,2),~(2,1)\}$, $S_5=\{(1,1,1,1,1)\}$ and $S_k$ is empty when $k=1,3,4$ and $k\ge 6$.  
Hence, we obtain 
\begin{align*}
\frac{CC_{3,270}}{\Pi(270)}&=\left((-L_3)^2\frac{(-1)^3\cdot 2}{L_4 L_5}+(-L_3)^5\frac{(-1)^5}{L_4^5}\right)\\
&=\frac{L_3^5}{L_4^5}-\frac{2 L_3^2}{L_4 L_5}\,. 
\end{align*} 
In fact, 
\begin{align*} 
\sum_{n=0}^\infty\frac{CC_{3,n}}{\Pi(n)}x^n&=\dfrac{-\dfrac{x^{27}}{L_3}}{\sum_{i=3}^\infty(-1)^i\dfrac{x^{3^i}}{L_i}}\\
&=1+\frac{L_3}{L_4}x^{54}+\frac{L_3^2}{L_4^2}x^{108}+\frac{L_3^3}{L_4^3}x^{162}\\
&\quad+\left(\frac{L_3^5}{L_4^5}-\frac{2 L_3^2}{L_4 L_5}\right)x^{270}+\left(\frac{L_3^6}{L_4^6}-\frac{3 L_3^3}{L_4^2 L_5}\right)x^{324}+\cdots\,. 
\end{align*} 
\bigskip

We can express the hypergeometric Cauchy numbers in terms of the binomial coefficients too. 
In fact, by using Lemma \ref{quotientrules} (\ref{quotientrule2}) 
instead of Lemma \ref{quotientrules} (\ref{quotientrule1}) 
in the proof of Theorem \ref{th_hgcauchycarlitz1}, we obtain the following:  

\begin{Prop}  
For $n\ge 1$, 
$$
CC_{N,n}=\Pi(n)\sum_{k=1}^{n}\binom{n+1}{k+1}(-L_N)^k\sum_{i_1,\dots,i_k\ge 0\atop r^{N+i_1}+\cdots+r^{N+i_k}=n+k r^N}\frac{(-1)^{i_1+\cdots+i_k}}{L_{N+i_1}\cdots L_{N+i_k}}\,.
$$ 
\label{th_hgcauchycarlitz2}  
\end{Prop}

\section{Incomplete Stirling-Carlitz numbers}  

In \cite{KK}, as analogues of the Stirling numbers of the first kind $\stf{n}{k}$ defined by 
\begin{equation} 
\frac{\bigl(-\log(1-t)\bigr)^k}{k!}=\sum_{n=0}^\infty\stf{n}{k}\frac{t^n}{n!}\,,
\label{stf}  
\end{equation}
the Stirling-Carlitz numbers of the first kind  $\stf{n}{k}_C$ were introduced by 
\begin{equation}  
\frac{\bigl(\log_C(z)\bigr)^k}{\Pi(k)}=\sum_{n=0}^\infty\stf{n}{k}_C\frac{z^n}{\Pi(n)}\,. 
\label{stfcarlitz}  
\end{equation}  
As analogues of the Stirling numbers of the second kind $\sts{n}{k}$ defined by
$$
\frac{(e^t-1)^k}{k!}=\sum_{n=0}^\infty\sts{n}{k}\frac{t^n}{n!}\,,
$$ 
the Stirling-Carlitz numbers of the second kind $\sts{n}{k}_C$ were introduced by 
\begin{equation}  
\frac{\bigl(e_C(z)\bigr)^k}{\Pi(k)}=\sum_{n=0}^\infty\sts{n}{k}_C\frac{z^n}{\Pi(n)}\,. 
\label{stscarlitz}  
\end{equation}  
By the definition (\ref{stfcarlitz}), we have 
\begin{equation}
\stf{n}{0}_C=0\quad(n\ge 1),\quad \stf{n}{m}_C=0\quad(n<m)\quad  
\hbox{and}\quad \stf{n}{n}_C=1\quad(n\ge 0) 
\label{eqn:aabb} 
\end{equation}
and 
\begin{equation}
\sts{n}{0}_C=0\quad(n\ge 1),\quad \sts{n}{m}_C=0\quad(n<m)\quad  
\hbox{and}\quad \sts{n}{n}_C=1\quad(n\ge 0)\,.
\label{eqn:aacc} 
\end{equation}

On the other hand, in \cite{Cha, Ko6, KLM, KMS}, so-called incomplete Stirling numbers of the fist kind and of the second kind were introduced as some generalizations of the classical Stirling numbers of the fist kind and of the second kind.  One of the incomplete Stirling numbers is {\it restricted Stirling number}, and another is {\it associated Stirling number}.  
Associated Stirling numbers of the second kind $\sts{n}{k}_{\ge m}$ are given by
\begin{equation} 
\frac{\bigl(e^x-E_{m-1}(x)\bigr)^k}{k!}=\sum_{n=0}^\infty\sts{n}{k}_{\ge m}\frac{x^n}{n!}\quad(m\ge 1)\,, 
\label{assoc.sts} 
\end{equation} 
where 
$$
E_m(x)=\sum_{n=0}^m\frac{x^n}{n!}\,. 
$$     
When $m=1$, $\sts{n}{k}=\sts{n}{k}_{\ge 1}$ is the classical Stirling numbers of the second kind.  
Restricted Stirling numbers of the second kind $\sts{n}{k}_{\ge m}$ are given by
\begin{equation} 
\frac{\bigl(E_{m}(x)-1\bigr)^k}{k!}=\sum_{n=0}^\infty\sts{n}{k}_{\le m}\frac{x^n}{n!}\quad(m\ge 1)\,. 
\label{rest.sts} 
\end{equation} 
When $m\to\infty$, $\sts{n}{k}=\sts{n}{k}_{\le\infty}$ is the classical Stirling numbers of the second kind. 

Associated Stirling numbers of the first kind $\stf{n}{k}_{\ge m}$ are given by
\begin{equation} 
\frac{\bigl(-\log(1-x)+F_{m-1}(-x)\bigr)^k}{k!}=\sum_{n=0}^\infty\stf{n}{k}_{\ge m}\frac{x^n}{n!}\quad(m\ge 1)\,, 
\label{assoc.stf} 
\end{equation} 
where 
$$
F_m(t)=\sum_{k=1}^m(-1)^{k+1}\frac{t^k}{k}\,. 
$$     
When $m=1$, $\stf{n}{k}=\stf{n}{k}_{\ge 1}$ is the classical Stirling numbers of the first kind.  
Restricted Stirling numbers of the first kind $\stf{n}{k}_{\ge m}$ are given by
\begin{equation} 
\frac{\bigl(-F_m(-x)\bigr)^k}{k!}=\sum_{n=0}^\infty\stf{n}{k}_{\le m}\frac{x^n}{n!}\quad(m\ge 1)\,. 
\label{rest.stf} 
\end{equation} 
When $m\to\infty$, $\stf{n}{k}=\stf{n}{k}_{\le\infty}$ is the classical Stirling numbers of the first kind.

Now, we introduce {\it associated Stirling-Carlitz numbers} and {\it restricted Stirling-Carlitz numbers}.  
The partial sum of the Carlitz exponential is denoted by  
$$
\mathcal E_m(x)=\sum_{i=0}^m\frac{x^{r^i}}{D_i}\,. 
$$
The associated Stirling-Carlitz numbers of the second kind  $\sts{n}{k}_{C, \ge m}$ are defined by 
\begin{equation}  
\frac{\bigl(e_C(z)-\mathcal E_{m-1}(z)\bigr)^k}{\Pi(k)}=\sum_{n=0}^\infty\sts{n}{k}_{C,\ge m}\frac{z^n}{\Pi(n)}\,. 
\label{assocstscarlitz}  
\end{equation}  
The restricted Stirling-Carlitz numbers of the second kind  $\sts{n}{k}_{C,\le m}$ are defined by 
\begin{equation}  
\frac{\bigl(\mathcal E_m(z)\bigr)^k}{\Pi(k)}=\sum_{n=0}^\infty\sts{n}{k}_{C,\le m}\frac{z^n}{\Pi(n)}\,. 
\label{reststscarlitz}  
\end{equation}  
When $m=0$ in (\ref{assocstscarlitz}) or $m\to\infty$ in (\ref{reststscarlitz}), $\sts{n}{k}_{C}=\sts{n}{k}_{C, \ge 0}=\sts{n}{k}_{C, \le\infty}$ is the original Stirling-Carlitz number of the second kind.  
The partial sum of the Carlitz logarithm is denoted by 
$$
\mathcal F_m(x)=\sum_{i=0}^m(-1)^i\frac{x^{r^i}}{L_i}\,. 
$$
The associated Stirling-Carlitz numbers of the first kind  $\stf{n}{k}_{C, \ge m}$ are defined by 
\begin{equation}  
\frac{\bigl(\log_C(z)-\mathcal F_{m-1}(z)\bigr)^k}{\Pi(k)}=\sum_{n=0}^\infty\stf{n}{k}_{C,\ge m}\frac{z^n}{\Pi(n)}\,. 
\label{assocstfcarlitz}  
\end{equation}  
The restricted Stirling-Carlitz numbers of the first kind  $\stf{n}{k}_{C,\le m}$ are defined by 
\begin{equation}  
\frac{\bigl(\mathcal F_m(z)\bigr)^k}{\Pi(k)}=\sum_{n=0}^\infty\stf{n}{k}_{C,\le m}\frac{z^n}{\Pi(n)}\,. 
\label{reststfcarlitz}  
\end{equation}  
When $m=0$ in (\ref{assocstfcarlitz}) or $m\to\infty$ in (\ref{reststfcarlitz}), $\stf{n}{k}_{C}=\stf{n}{k}_{C, \ge 0}=\stf{n}{k}_{C, \le\infty}$ is the original Stirling-Carlitz number of the first kind.  
\bigskip

Due to associated Stirling-Carlitz numbers of the second kind in (\ref{assocstscarlitz}), we can obtain a more explicit expression of hypergeometric Bernoulli-Carlitz numbers, expressed in Theorem \ref{th_hgbernoullicarlitz1} or Proposition \ref{th_hgbernoullicarlitz2}.

\begin{theorem} 
For $N\ge 1$ and $n\ge 1$, we have  
$$
BC_{N,n}=
\Pi(n)\sum_{k=1}^{n}\binom{n+1}{k+1}\frac{(-D_N)^k\Pi(k)}{\Pi(n+k r^N)}\sts{n+k r^N}{k}_{C,\ge N}\,. 
$$ 
\label{th27} 
\end{theorem}

\begin{proof} 
From (\ref{assocstscarlitz}), we have 
\begin{align*} 
\left(\sum_{j=0}^\infty\frac{x^{r^{N+j}-r^N}}{D_{N+j}}\right)^k&=\left(\frac{e_C(x)-\mathcal E_{N-1}(x)}{x^{r^N}}\right)^k\\
&=\sum_{n=k}^\infty\frac{\Pi(k)}{\Pi(n)}\sts{n}{k}_{C,\ge N}x^{n-k r^N}\\
&=\sum_{n=-(r^N-1)k}^\infty\frac{\Pi(k)}{\Pi(n+k r^N)}\sts{n+k r^N}{k}_{C,\ge N}x^n\,. 
\end{align*} 
Notice that 
\begin{align*}
\left.H^{(e)}\left(\frac{e_C(x)-\mathcal E_{N-1}(x)}{x^{r^N}}\right)\right|_{x=0}&=\left.\sum_{j=0}^\infty\frac{1}{D_{N+j}}\binom{r^{N+j}-r^N-e}{e}x^{r^{N+j}-r^N-e}\right|_{x=0}\\
&=
\begin{cases} 
\frac{1}{D_{N+i}}&\text{if $r^{N+i}-r^N=e$};\\
0&\text{otherwise}.
\end{cases} 
\end{align*}  
Applying Lemma \ref{productrule2} with 
$$
f_1(t)=\cdots=f_k(t)=\frac{e_C(x)-\mathcal E_{N-1}(x)}{x^{r^N}}\,, 
$$
we get 
\begin{equation}  
\frac{\Pi(k)}{\Pi(n+k r^N)}\sts{n+k r^N}{k}_{C,\ge N}=\sum_{i_1,\dots,i_k\ge 0\atop r^{N+i_1}+\cdots+r^{N+i_k}=n+k r^N}\frac{1}{D_{N+i_1}\cdots D_{N+i_k}}\,. 
\label{eq20} 
\end{equation}   
Together with Proposition \ref{th_hgbernoullicarlitz2}, we can get the desired result.  
\end{proof}  
\bigskip

{\bf Example.} 
Let $r=3$, $N=2$ and $n=18$. Comparing the coefficient of $x^n$ on both sides of
$$
\sum_{n=0}^\infty\frac{\Pi(k)}{\Pi(n)}\sts{n}{k}_{C,\ge 2}x^n=\left(\frac{x^9}{D_2}+\frac{x^{27}}{D_3}+\frac{x^{81}}{D_4}+\cdots\right)^k\,,
$$ 
for $k=1,2,\dots,18$, we have 
$$
\frac{\Pi(k)}{\Pi(18+9 k)}\sts{18+9 k}{k}_{C<\ge 2}=\frac{k}{D_2^{k-1}D_3}\,. 
$$ 
Hence, 
\begin{align*}
\frac{BC_{2,18}}{\Pi(18)}&=\sum_{k=1}^{18}\binom{19}{k+1}(-D_2)^k\frac{k}{D_2^{k-1}D_3}\\
&=\frac{D_2}{D_3}\sum_{k=1}^{18}(-1)^k\binom{19}{k+1}k
=-\frac{D_2}{D_3}\,. 
\end{align*} 
\bigskip

Bernoulli-Carlitz numbers can be expressed in term of the Stirling-Carlitz numbers of the second kind: 
$$
BC_n=\sum_{j=0}^\infty\frac{(-1)^j D_j}{L_j^2}\sts{n}{r^j-1}_C
$$ 
(\cite[Theorem 2]{KK}).  
When $N=0$, Theorem \ref{th27} is reduced to a different expression of Bernolli-Carlitz numbers in terms of the Stirling-Carlitz numbers of the second kind.  

\begin{Cor} 
For $n\ge 1$, we have 
$$
BC_n=\Pi(n)\sum_{k=1}^{n}\binom{n+1}{k+1}\frac{(-1)^k\Pi(k)}{\Pi(n+k)}\sts{n+k}{k}_C\,. 
$$
\end{Cor} 

\noindent 
{\it Remark.}  
This is an analogue of 
$$
B_n=\sum_{k=1}^n\dfrac{(-1)^{k}\binom{n+1}{k+1}}{\binom{n+k}{k}}\sts{n+k}{k}\,,
$$
which is a simple formula appeared in \cite{Gould,SS}. 

%%%%%%%%%%%%%%%%%%%%%
\bigskip  

Similarly, due to associated Stirling-Carlitz numbers of the first kind in (\ref{assocstfcarlitz}), we can obtain a more explicit expression of hypergeometric Cauchy-Carlitz numbers, expressed in Theorem \ref{th_hgcauchycarlitz1} or Proposition \ref{th_hgcauchycarlitz2}.

\begin{theorem} 
For $N\ge 1$ and $n\ge 1$, we have  
$$
CC_{N,n}=
\Pi(n)\sum_{k=1}^{n}\binom{n+1}{k+1}\frac{(-1)^{N k}(-L_N)^k\Pi(k)}{\Pi(n+k r^N)}\stf{n+k r^N}{k}_{C,\ge N}\,. 
$$ 
\label{th28} 
\end{theorem}

\begin{proof} 
From (\ref{assocstfcarlitz}), we have 
\begin{align*} 
\left(\sum_{j=0}^\infty\frac{(-1)^{N+j}x^{r^{N+j}-r^N}}{L_{N+j}}\right)^k&=\left(\frac{\log_C(x)-\mathcal F_{N-1}(x)}{x^{r^N}}\right)^k\\
&=\sum_{n=k}^\infty\frac{\Pi(k)}{\Pi(n)}\stf{n}{k}_{C,\ge N}x^{n-k r^N}\\
&=\sum_{n=-(r^N-1)k}^\infty\frac{\Pi(k)}{\Pi(n+k r^N)}\stf{n+k r^N}{k}_{C,\ge N}x^n\,. 
\end{align*} 
Notice that 
\begin{align*}
\left.H^{(e)}\left(\frac{\log_C(x)-\mathcal F_{N-1}(x)}{x^{r^N}}\right)\right|_{x=0}&=\left.\sum_{j=0}^\infty\frac{(-1)^{N+j}}{L_{N+j}}\binom{r^{N+j}-r^N-e}{e}x^{r^{N+j}-r^N-e}\right|_{x=0}\\
&=
\begin{cases} 
\frac{(-1)^{N+i}}{L_{N+i}}&\text{if $r^{N+i}-r^N=e$};\\
0&\text{otherwise}.
\end{cases} 
\end{align*}  
Applying Lemma \ref{productrule2} with 
$$
f_1(t)=\cdots=f_k(t)=\frac{\log_C(x)-\mathcal F_{N-1}(x)}{x^{r^N}}\,, 
$$
we get 
\begin{equation}  
\frac{\Pi(k)}{\Pi(n+k r^N)}\stf{n+k r^N}{k}_{C,\ge N}=\sum_{i_1,\dots,i_k\ge 0\atop r^{N+i_1}+\cdots+r^{N+i_k}=n+k r^N}\frac{(-1)^{N k+i_1+\cdots+i_k}}{L_{N+i_1}\cdots L_{N+i_k}}\,. 
\label{eq202} 
\end{equation}   
Together with Proposition \ref{th_hgcauchycarlitz2}, we can get the desired result.  
\end{proof}  
\bigskip 

\noindent 
{\bf Example.}  
Let $r=3$, $N=3$ and $n=270$. 
Comparing the coefficient of $x^n$ on both sides of
$$
\sum_{n=0}^\infty\frac{\Pi(k)}{\Pi(n)}\stf{n}{k}_{C,\ge 3}x^n=\left(-\frac{x^{27}}{L_3}+\frac{x^{81}}{L_4}-\frac{x^{243}}{L_5}+\frac{x^{729}}{L_6}-\cdots\right)^k\,,
$$ 
for $k=1,2,3,4$, we have 
$$
\frac{\Pi(k)}{\Pi(270+27 k)}\stf{270+27 k}{k}_{C<\ge 3}=(-1)^{k-1}\frac{k(k-1)}{L_3^{k-2}L_4 L_5} . 
$$ 
and for $k=5,6,\dots,270$, we have 
$$
\frac{\Pi(k)}{\Pi(270+27 k)}\stf{270+27 k}{k}_{C<\ge 3}=(-1)^{k-1}\frac{\binom{k}{5}}{L_3^{k-5}L_4^5}
+(-1)^{k-1}\frac{k(k-1)}{L_3^{k-2}L_4 L_5}\,. 
$$ 
Therefore, 
\begin{align*}
\frac{CC_{3,270}}{\Pi(270)}&=\sum_{k=1}^{270}\binom{271}{k+1}(-1)^{3 k}(-L_3)^k\frac{(-1)^{k-1}k(k-1)}{L_3^{k-2}L_4 L_5}\\
&\quad +\sum_{k=5}^{270}\binom{271}{k+1}(-1)^{3 k}(-L_3)^k\frac{(-1)^{k-1}\binom{k}{5}}{L_3^{k-5}L_4^5}\\
&=\frac{L_3^2}{L_4 L_5}\sum_{k=1}^{270}(-1)^{k-1}k(k-1)\binom{271}{k+1}
+\frac{L_3^5}{L_4^5}\sum_{k=5}^{270}(-1)^{k-1}\binom{271}{k+1}\binom{k}{5}\\
&=-\frac{2 L_3^2}{L_4 L_5}+\frac{L_3^5}{L_4^5}\,. 
\end{align*}

\bigskip 

Cauchy-Carlitz numbers can be expressed in term of the Stirling-Carlitz numbers of the first kind: 
$$
CC_n=\sum_{j=0}^\infty\frac{1}{L_j}\stf{n}{r^j-1}_C
$$ 
(\cite[Theorem 1]{KK}).  
When $N=0$, Theorem \ref{th28} is reduced to a different expression of Cauchy-Carlitz numbers in terms of the Stirling-Carlitz numbers of the first kind.  

\begin{Cor} 
For $n\ge 1$, we have 
$$
CC_n=\Pi(n)\sum_{k=1}^{n}\binom{n+1}{k+1}\frac{(-1)^k\Pi(k)}{\Pi(n+k)}\stf{n+k}{k}_C\,. 
$$
\end{Cor} 

\noindent 
{\it Remark.} 
This is an analogue of 
$$
c_n=\sum_{k=1}^n\dfrac{(-1)^{n-k}\binom{n+1}{k+1}}{\binom{n+k}{k}}\stf{n+k}{k}\,,
$$
which is Proposition 2 in \cite{KK}.  

%%%%%%%%%%%%%%%%%%%%%%%% 

\end{document}